\documentclass{amsart}
\usepackage{amssymb}
\usepackage{verbatim}
\begin{document}
%\initfloatingfigs
%\input macro.tex
\newcommand\im{\operatorname{Im}}
\newcommand\re{\operatorname{Re}}
\newcommand\Id{\operatorname{Id}}
\newcommand\Real{\mathbb{R}}
\newcommand\RR{\mathbb{R}}
\newcommand\CC{\mathbb{C}}
\newcommand\BB{\mathbb{B}}
\newcommand\Cx{\mathbb{C}}
\newcommand\sphere{\mathbb{S}}
\newcommand\Cinf{{\mathcal C}^{\infty}}
\newcommand\dist{{\mathcal C}^{-\infty}}
\newcommand\CI{C^\infty}
\newcommand\dCI{\dot C^\infty}
\newcommand\Diff{\operatorname{Diff}}
\newcommand\pa{\partial}
\newcommand\supp{\operatorname{supp}}
\newcommand\ep{\epsilon}
\newcommand\bl{{\mathrm b}}
\newcommand{\sH}{\mathsf{H}}
\newcommand\Vf{\mathcal{V}}
\newcommand\Vb{\mathcal{V}_\bl}
\newcommand\Hom{\mathrm{Hom}}
\newcommand\Ker{\mathrm{Ker}}
\newcommand\Ran{\mathrm{Ran}}
\newcommand\Diffb{\mathrm{Diff}_\bl}
\newcommand\Psib{\Psi_\bl}
\newcommand\Psibc{\Psi_{\mathrm{bc}}}
\newcommand\Tb{{}^{\bl} T}
\newcommand\Sb{{}^{\bl} S}

\newcommand\scl{{\mathrm{sc}}}
\newcommand\Tsc{{}^{\scl} T}
\newcommand\Vsc{\mathcal{V}_\scl}
\newcommand\Diffsc{\mathrm{Diff}_\scl}
\newcommand\Hsc{H_\scl}

\newcommand\cL{\mathcal L}

\renewcommand{\Im}{\operatorname{Im}}
\renewcommand{\Re}{\operatorname{Re}}

\newcommand{\mar}[1]{{\marginpar{\sffamily{\scriptsize #1}}}}
\newcommand\av[1]{\mar{AV:#1}}
\newcommand\jw[1]{\mar{JW:#1}}

\newcommand{\abs}[1]{{\left\lvert{#1}\right\rvert}}
\newcommand{\norm}[1]{{\left\lVert{#1}\right\rVert}}
\newcommand{\ang}[1]{{\left\langle{#1}\right\rangle}}

\newcommand{\g}{\mathsf{g}}
\newcommand{\A}{{\mathcal{A}}}
\newcommand{\B}{{\mathcal{B}}}

\newcommand{\tpsi}{\tilde{\psi}}

\setcounter{secnumdepth}{3}
\newtheorem{lemma}{Lemma}[section]
\newtheorem{prop}[lemma]{Proposition}
\newtheorem{thm}[lemma]{Theorem}
\newtheorem{cor}[lemma]{Corollary}
\newtheorem{result}[lemma]{Result}
\newtheorem*{thm*}{Theorem}
\newtheorem*{prop*}{Proposition}
\newtheorem*{cor*}{Corollary}
\newtheorem*{conj*}{Conjecture}
\numberwithin{equation}{section}
\theoremstyle{remark}
\newtheorem{rem}[lemma]{Remark}
\newtheorem*{rem*}{Remark}
\theoremstyle{definition}
\newtheorem{Def}[lemma]{Definition}
\newtheorem*{Def*}{Definition}
\renewcommand{\theenumi}{\roman{enumi}}
\renewcommand{\labelenumi}{(\theenumi)}

\title{Positive commutators at the bottom of the spectrum}
\author[Andras Vasy]{Andr\'as Vasy}
\author{Jared Wunsch}
\address{Department of Mathematics, Stanford University}
\address{Department of Mathematics, Northwestern University}
\email{andras@math.stanford.edu}
\email{jwunsch@math.northwestern.edu}
\thanks{The authors thank Rafe Mazzeo for helpful conversations and
  pointing out the reference \cite{Mazzeo-McOwen}.  They gratefully
  acknowledge partial support from the NSF under grant numbers  
  DMS-0801226 (AV) and DMS-0700318 (JW)}
\date{\today}

\begin{abstract}
Bony and H\"afner have recently obtained positive commutator estimates on the Laplacian in
the low-energy limit on asymptotically Euclidean spaces; these
estimates can be used to prove local energy decay estimates if the
metric is non-trapping.  We simplify the proof of the estimates of
Bony-H\"afner and
generalize them to the setting of scattering manifolds (i.e.\
manifolds with large conic ends), by applying a sharp Poincar\'e
inequality.  Our main result is the positive commutator estimate
$$
\chi_I(H^2\Delta_g)\frac{i}{2}[H^2\Delta_g,A]\chi_I(H^2\Delta_g)
\geq C\chi_I(H^2\Delta_g)^2,
$$
where $H\uparrow \infty$ is a \emph{large} parameter, $I$ is a compact
interval in $(0,\infty),$ and $\chi_I$ its indicator function, and
where $A$ is a differential operator supported outside a compact set
and equal to
$(1/2)(r D_r +(r D_r)^*)$ near infinity.  The Laplacian can also be
modified by the addition of a positive potential of sufficiently
rapid decay---the same estimate then holds for the resulting
Schr\"odinger operator.
\end{abstract}

\maketitle

\section{Introduction}
The purpose of this paper is to clarify an intricate argument recently
introduced by Bony and H\"afner \cite{Bony-Haefner1} and use these ideas to
generalize certain of the results of \cite{Bony-Haefner1}.  The
central thrust of \cite{Bony-Haefner1} is first of all to obtain
certain kinds of commutator estimates for the the Laplacian and its
square root on asymptotically Euclidean space.  The authors then
employ those estimates to yield energy decay results for the wave
equation, and, ultimately, global existence results for quadratically
semilinear wave equations on these spaces.  In a subsequent note
\cite{Bony-Haefner2}, applications of the linear results to the low
frequency limiting absorption principle were shown.  The novel tool
central to all of these applications is the commutator estimate
\begin{equation}\label{BHest}
\chi_I(H^2\Delta_g)\frac{i}{2}[H^2\Delta_g,A]\chi_I(H^2\Delta_g)
\geq C\chi_I(H^2\Delta_g)^2,
\end{equation}
where $H\uparrow \infty$ is a \emph{large} parameter, $I$ is a compact
interval in $(0,\infty),$ and $\chi_I$ its indicator function, and
where $A$ is a differential operator supported outside a compact set
and equal to
$(1/2)(r D_r +(r D_r)^*)$ near infinity.  The estimate
\eqref{BHest} is thus a low-energy version of the positive commutator
construction that is ubiquitous in scattering theory; we remark that
the analogous \emph{high}-energy estimate would not be true with this
choice of $A,$ supported outside a compact set: by standard results in
microlocal analysis, the symbol of $A$ would have to be strictly
increasing along all geodesics, lifted to the cotangent bundle.
Indeed, on a manifold with trapped geodesics, the construction of such
a high-energy commutant is manifestly impossible.

In \cite{Bony-Haefner1}, the estimate \eqref{BHest} is proved by a
multi-step process involving a sequence of perturbation arguments,
starting from flat $\RR^n.$ It is thus a priori unclear whether such
estimates continue to hold if we vary the topology of our space and
its end structure.  In this paper we show that \eqref{BHest} (as well
as a related estimate for $\sqrt{\Delta}$) does indeed continue to hold
on any long-range metric perturbation of a \emph{scattering manifold},
and further holds even if a short range (in a suitable sense) non-negative
potential is added.
The class of scattering manifolds, introduced by Melrose
\cite{RBMSpec}, consists of all manifolds with ends that look
asymptotically like the large ends of cones.  The topology of interior
and of the cross sections of the ends is unrestricted.  Our methods
are nonperturbative and simple, involving only commutator estimates
for differential operators and a sharp Poincar\'e-type inequality on
these manifolds.  We anticipate that these methods will prove quite
flexible in the investigation of energy decay in a variety of other
asymptotic geometries.

We do not explore the applications of our estimate in detail here, as
the methods of \cite{Bony-Haefner1} apply, mutatis mutandis, directly
to our situation.  We content ourselves with restating the energy
decay estimate of \cite{Bony-Haefner1} for solutions to the wave
equation in the final section of the paper and sketching the main
ingredients in its proof, adapted to our setting.  This estimate applies on
scattering manifolds with no trapped geodesics.\footnote{Such a
  manifold must in fact be contractible, but we note that even $\RR^n$
  can be equipped with scattering metrics different from the round
  metric on the sphere at infinity, so this result remains broader
  than that of \cite{Bony-Haefner1}.}

We point out here that Guillarmou and Hassell started an extensive and
very detailed study of the Laplacian on scattering manifolds near the
bottom of the spectrum, \cite{Guillarmou-Hassell:Resolvent-I}, with a
particular emphasis on the Schwartz kernel of the resolvent of the
Laplacian on a resolved space.  Our methods give the estimates we need
more quickly, but naturally the results of
\cite{Guillarmou-Hassell:Resolvent-I} give more detail on the
resolvent kernel, which in principle implies for instance results on
the energy decay\footnote{Note, however, that $L^2$-based estimates
  are not always easy to get from precise description of the Schwartz
  kernel!}.  We also remark that Bouclet \cite{Bouclet1} has recently
proved weighted low-energy estimates generalized those of
\cite{Bony-Haefner1} for \emph{powers} of the resolvent on an asymptotically
Euclidean space.

Our paper is structured as follows. In Section~\ref{sec:b-sc} we recall
the background material concerning b- (or totally characteristic)
and scattering differential operators. In Section~\ref{sec:Poincare}
we obtain Poincar\'e inequalities and in Section~\ref{sec:weights} weighted
diferential estimates that we use in Section~\ref{sec:Mourre}
to prove our positive commutator estimate. Finally, in Section~\ref{sec:wave}
we show how these results can be applied to study energy decay for the
wave equation, following the method of Bony and H\"afner \cite{Bony-Haefner1}.

\section{b- and scattering geometry}\label{sec:b-sc}
We very briefly recall the basic definitions of the b- and scattering
structures on an $n$-dimensional manifolds with boundary, denoted $X$;
we refer to \cite{RBMSpec} for more detail.
A boundary defining function $x$ on $X$ is a non-negative $\CI$ function
on $X$ whose zero set is exactly $\pa X$, and whose differential
does not vanish there. We recall that $\dCI(X)$, which may also be
called the set of Schwartz functions, is the subset of
$\CI(X)$ consisting of functions vanishing at the boundary with
all derivatives, the dual of $\dCI(X)$ is tempered distributional densities
$\dist(X;\Omega X)$; tempered distributions $\dist(X)$ are elements of
the dual of Schwartz densities, $\dCI(X;\Omega X)$.

Let $\Vf(X)$ be the Lie algebra of all $\CI$ vector fields on $X$;
thus $\Vf(X)$ is the set of all $\CI$ sections of $TX$. In local
coordinates $(x,y_1,\ldots,y_{n-1})$,
$$
\pa_x,\pa_{y_1},\ldots,\pa_{y_{n-1}}
$$
form a local basis for $\Vf(X)$, i.e.\ restrictions of elements
of $\Vf(X)$ to the coordinate chart can be expressed uniquely as
a linear combination of these vector fields with $\CI$ coefficients.
We next define $\Vb(X)$ to be the Lie algebra of $\CI$ vector fields tangent
to $\pa X$; in local coordinates 
$$
x\pa_x,\pa_{y_1},\ldots,\pa_{y_{n-1}}
$$
form a local basis in the same sense. Thus, $\Vb(X)$ is the set
of all $\CI$ sections of a bundle, called the b-tangent bundle of
$X$, denoted $\Tb X$. Finally, $\Vsc(X)=x\Vb(X)$ is the Lie algebra
of scattering vector fields;
$$
x^2\pa_x,x\pa_{y_1},\ldots,x\pa_{y_{n-1}}
$$
form a local basis now. Again, $\Vsc(X)$ is the set of 
all $\CI$ sections of a bundle, called the scattering tangent bundle of
$X$, denoted $\Tsc X$.

The dual bundles of $TX,\Tb X,\Tsc X$ are $T^*X,\Tb^*X,\Tsc^*X$ respectively,
with local bases
$$
dx,\ dy_j,\ \text{resp.}\ \frac{dx}{x},\ dy_j,\ \text{resp.}\ \frac{dx}{x^2},
\ \frac{dy_j}{x},\ j=1,\ldots,n-1.
$$
These induce form bundles and density bundles as usual. In particular,
local bases of the density bundles are
$$
|dx\,dy_1\ldots dy_{n-1}|\ \text{resp.}\ x^{-1}|dx\,dy_1\ldots dy_{n-1}|,
\ \text{resp.}\ x^{-n-1}|dx\,dy_1\ldots dy_{n-1}|.
$$
If $X$ is compact, the $L^2$-spaces relative to
these classes of densities
are well-defined as Banach spaces, up to equivalence
of norms; they are denoted by $L^2(X)$, $L^2_{\bl}(X)$, $L^2_{\scl}(X)$,
respectively.

The classes of vector fields mentioned induce algebras of differential
operators, consisting of locally finite sums of products of these
vector fields and elements of $\CI(X)$, considered as operators on $\CI(X)$.
These are denoted by $\Diff(X)$, $\Diffb(X)$ and $\Diffsc(X)$, respectively.
These in turn give rise to (integer order) Sobolev spaces. Thus,
for $m\geq 0$
integer,
$$
H_{\bullet}^m(X)=\{u\in L^2_{\bullet}(X):\ Qu\in L^2_{\bullet}(X)
\ \forall Q\in\Diff_{\bullet}^m(X)\},
$$
where $\bullet$ is either $\bl$ or $\scl$ and
where $Qu$ is a priori defined as a (tempered) distribution.

% We also write
% $$
% H_{\bl}^{m,l}(X)=x^l H_{\bl}^m(X).
% $$
%% I couldn't find any occurrences!

A similar construction leads to symbol classes:
We let $S^k(X),$ the space of symbols of order $k$, consist of
functions $f$ such that
$$
x^k Lf\in L^\infty(X)\ \text{for all}\ L\in\Diffb(X).
$$
We note, in particular, that
$$
x^\rho\CI(X)\subset S^{-\rho}(X)
$$
since $\Diffb(X)\subset\Diff(X)$.
As $\Diffb(X)$ (a priori acting, say, on tempered distributions)
preserves $S^k(X)$, and one can extend $\Diffb(X)$ and $\Diffsc(X)$
by `generalizing the coefficients':
$$
S^k\Diffb^m(X)=\{\sum_j a_j Q_j:\ a_j\in S^k(X),\ Q_j\in\Diffb^m(X)\},
$$
with the sum being locally finite, and defining $S^k\Diffsc^m(X)$
similarly. In particular,
$$
x^k\Diffb^m(X)\subset S^{-k}\Diffb^m(X),
\ x^k\Diffsc^m(X)\subset S^{-k}\Diffsc^m(X).
$$
Then $Q\in S^{k}\Diffsc^m(X)$, $Q'\in S^{k'}\Diffsc^{m'}(X)$ gives
$QQ'\in S^{k+k'}\Diffsc^{m+m'}(X)$, and the analogous statement
for $S^k\Diffb^m(X)$ also holds.

An example of particular interest is the radial, or geodesic,
compactification of $\RR^n$, which compactifies $\RR^n$ as a closed
ball, $X=\overline{\BB^n}$; see \cite[Section~1]{RBMSpec} for an extended
discussion, with the compactification called stereographic
compactification there. In this case, the set of Schwartz functions on $\RR^n$
lifts to $\dCI(X)$ (justifying the `Schwartz' terminology for the latter),
the set of 0th order classical symbols
on $\RR^n$, i.e.\ 0th order symbols $a$ with an asymptotic expansion
$a(r\omega)\sim \sum_{j=0}^\infty r^{-j} a_j(\omega)$ in polar coordinates,
lifts to $\CI(X)$, the translation invariant
vector fields on $\RR^n$ lift to a basis of $\Vsc(X)$, and $\Hsc^m(X)$
is the standard Sobolev space $H^m(\RR^n)$ (under the natural identification
of functions), while $S^k(X)$ is the
standard symbol space $S^k(\RR^n)$. One way of seeing these statements
is to introduce `inverse polar coordinates' $z=x^{-1}\omega$, $x\in(0,1)$,
$\omega\in\sphere^{n-1}$, in the exterior of a closed ball in
$\RR^n_z$, and use polar coordinates $(\rho,\omega)\in(1/2,1)
\times\sphere^{n-1}$
near $\pa\BB^n$, with $\BB^n$ considered as the unit ball in $\RR^n$; then
one suitable identification of the exterior of the ball of radius 2 in
$\RR^n_z$ with the interior
of a collar neighborhood of $\pa\BB^n$ in $\overline{\BB^n}$ is
$$
(0,1/2)\times\sphere^{n-1}\ni(x,\omega)\mapsto(\rho,\omega)=(1-x,\omega)
\in(1/2,1)\times\sphere^{n-1}.
$$

\section{Poincar\'e inequalities}\label{sec:Poincare}
Let $g$ be an scattering metric on a compact manifold
with boundary $X$ of dimension $n$, and $L^2_g(X)$ the metric $L^2$-space.
That is, as introduced by Melrose \cite{RBMSpec},
we assume that $g$ is a Riemannian metric on $X^\circ$, and
that $\pa X$ has a collar neighborhood $U$ and a boundary defining function
$x$ such that on $U$,
$$
g=\frac{dx^2}{x^4}+\frac{h}{x^2},
$$
where $h$ is a symmetric 2-cotensor, $h\in\CI(X;T^*X\otimes T^*X)$,
which restricts to a metric on $\pa X$. Then with $h_0=h|_{\pa X}$, and
also extended to $U$ using the product decomposition,
\begin{equation}\label{eq:asymp-Eucl}
g=\frac{dx^2}{x^4}+\frac{h_0}{x^2}+g_1,\ g_1\in x\CI(X;\Tsc^*X\otimes\Tsc^*X).
\end{equation}
Below we assume that $g$ is of this form, with merely
\begin{equation}\label{eq:weaker-asymp}
g_1\in S^{-\rho}(X;\Tsc^*X\otimes\Tsc^*X),\ \rho>0.
\end{equation}
Then the Laplacian $\Delta_g\in\Diffsc^2(X)$ satisfies
$\Delta_g\in x^2\Diffb^2(X)$, namely $\Delta_g=x^2\Delta_{\bl}$,
$\Delta_{\bl}\in\Diffb^2(X)$. Explicitly, as shown by Melrose
\cite[Proof of Lemma~3]{RBMSpec}, in local
coordinates $(x,y)$ on a collar neighborhood of $\pa X$,
\begin{equation}\label{eq:Delta-b-form}
\Delta_{\bl}=D_x x^2D_x+i(n-1)xD_x+\Delta_0+x^\rho R,\ R\in S^0\Diffb^2(X),
\end{equation}
where $\Delta_0$ is the Laplacian of the boundary metric.
Moreover, the density $|dg|=x^{-n}|dg_{\bl}|$,
where $|dg_{\bl}|$ is a non-degenerate b-density,
so $L^2_g(X)=x^{n/2}L^2_{\bl}(X)$.

We now recall a standard result in the b-calculus on the mapping
properties of $\Delta_g$.
Although we work with $\Delta_g+V$ with $V=0$ to be concise,
$V\in S^{-2-\rho}(X)$ with $V\geq 0$, $\rho>0$, can easily be accommodated.
We will in fact not use this result in the sequel,
but remark that its use eliminates the need for some of our
arguments (at the expense of b-machinery)
in sufficiently high dimension ($n\geq 5$)
\footnote{A proof of Lemma~\ref{lemma:Vb-est} proceeds as follows.
The statement of this lemma with `isomorphism'
replaced by `Fredholm of index 0' follows from
\cite[Lemma~2.1]{Guillarmou-Hassell:Resolvent-I} (which in turn essentially
quotes \cite{Melrose:Atiyah}), since, keeping in mind
that $\Delta=x^{n/2+1}P_b x^{-n/2+1}$ with the notation of that paper,
$P_b:x^{-1}H^2_\bl(X)\to xL^2_{\bl}(X)$ is shown to be Fredholm of index 0
there. By Lemma~2.2 of \cite{Guillarmou-Hassell:Resolvent-I} elements
of the nullspace of $\Delta$ would necessarily be in
$x^{n/2-1} H_{\bl}^\infty(X)$.
One deduces that $du\in L^2_\scl(X;\Tsc^*X)$, and a regularization argument
allows one to conclude from $\Delta u=0$ that $du=0$, and then that $u=0$.}.

\begin{lemma}\label{lemma:Vb-est}
Suppose $n\geq 5$. Then
$$
\Delta_g:x^{n/2-2}H_\bl^2(X)\to x^{n/2}L^2_\bl(X)=L^2_g(X)
$$
is an isomorphism.
In particular, for any $Q\in\Vb(X)$,
$$
\|x^2 u\|_{L^2_g(X)}+\|x^2Q u\|_{L^2_g(X)}\leq C\|\Delta_g u\|_{L^2(X)}.
$$
\end{lemma}

It is also useful to have the Poincar\'e inequality at our disposal.
This can be proved by b-techniques; we give an elementary proof.

\begin{lemma}\label{lemma:Poincare}
Suppose $l>1$, $l>l'$. Then for $u\in x^{l+1} H^1_{\bl}(X)$,
$$
\|x u\|_{x^{l'} L^2_\bl(X)}\leq C\|\nabla_g u\|_{x^l L^2_{\bl}(X)}.
$$
In particular, for $n\geq 3$, with $l=n/2$, $\ep=l-l'>0$,
$$
\|x^{1+\ep} u\|_{L^2_g(X)}\leq C\|\nabla_g u\|_{L^2_g(X)}.
$$
\end{lemma}

\begin{proof}
It suffices to prove this for $u\in\dCI(X)$ as both sides are continuous
on $x^{l+1}H^1_{\bl}(X)$. Moreover, it suffices to show that for such $u$,
$$
\|\chi x u\|_{x^{l'} L^2_\bl(X)}\leq C\|\nabla_g u\|_{x^l L^2_{\bl}(X)},
$$
$\chi\in\CI_c(X)$ supported in a collar neighborhood of $\pa X$, which
is then identified with $[0,x_0)_x\times\pa X$, for the rest will then follow
by the standard Poincar\'e inequality on $H^1_0(K)$ where $K\subset X^\circ$
is compact. This in turn follows from
$$
\|\chi x u\|_{x^{l'} L^2_\bl(X)}\leq C\|x^2D_x u\|_{x^l L^2_{\bl}(X)},
$$
i.e.
\begin{equation}\label{eq:bdy-Poincare}
\int \chi^2 |u|^2 x^{-2l'+1}\,dx\,dy\leq C^2\int |D_x u|^2 x^{-2l+3}\,dx\,dy.
\end{equation}
But in local coordinates near $\pa X$, for $k<1/2$, and for $x\leq x_0$
\begin{equation*}\begin{split}
|u(x,y)|&=\Big|\int_0^x (\pa_x u)(s,y)\,ds\Big|
=\Big|\int_0^x s^{k}(\pa_x u)(s,y) s^{-k}\,ds\Big|\\
&\leq \Big(\int_0^x s^{2k}|(\pa_x u)(s,y)|^2\,ds\Big)^{1/2}
\Big(\int_0^x s^{-2k}\,ds\Big)^{1/2}\\
&\leq \Big(\int_0^{x_0} s^{2k}|(\pa_x u)(s,y)|^2\,ds\Big)^{1/2}
C'x^{-k+1/2},
\end{split}\end{equation*}
and thus, provided $p-2k+1>-1$,
$$
\int_0^{x_0} x^p |u(x,y)|^2\,dx\leq C''
\int_0^{x_0} s^{2k}|(\pa_x u)(s,y)|^2\,ds.
$$
Integration with respect to $y$ now gives
$$
\int \chi^2 |u|^2 x^{p}\,dx\,dy\leq C^2\int |D_x u|^2 x^{2k}\,dx\,dy.
$$
So take $k=-l+3/2$, so $k<1/2$ is satisfied for $l>1$. Then let
$p=-2l'+1$, so $p-2k+1=-2l'+2l-1$, and $p-2k+1>-1$ is satisfied if $l'<l$.

\end{proof}

We now prove a sharp version of the Poincar\'e inequality; this
will follow from a weighted Hardy inequality, which can
be found in the Appendix of \cite{Mazzeo-McOwen}; we give a proof for
completeness:
\begin{lemma}\label{lemma:Hardy}
Let $u \in \dCI_c([0,\infty)),$ and let $d\mu=x^{-n-1} \, dx$ on $(0,\infty).$   If $s<(n-2)/2,$ we have
$$
\norm{x^{1+s} u}^2_{L^2(d\mu)}\leq \frac{4}{(n-2-2s)^2} \norm{x^{2+s}
  \pa_x u}^2_{L^2(d\mu)}.
$$
\end{lemma}
\begin{proof}
We follow the usual proof of Hardy's inequality, noting that one
usually uses $r=1/x$ as the independent variable.  We will use the
abbreviated notation $L^2_\mu= L^2 (d\mu).$

As
$$
[x^2\pa_x, x^{1+2s}] = (1+2s) x^{2+2s},
$$
pairing with $u$ yields
\begin{align*}
(1+2s)\norm{x^{1+s}u}_{L^2_\mu}^2 &=\ang{x^2 \pa_x (x^{1+2s}u),u}_{L^2_\mu} - \ang{x^{1+2s}
  x^2 \pa_x u,u}_{L^2_\mu}\\
&= \int_0^\infty \pa_x(x^{1+2s}u) x^{-n+1} \overline{u} \, dx-
\int_0^\infty x^{3+2s} u' \overline{u} x^{-n-1} \, dx
\end{align*}
Integrating by parts yields
$$
(1+2s)\norm{x^{1+s}u}_{L^2_\mu}^2 = (n-1) \int \abs{x^{1+s} u}^2 \, x^{-n-1} \,
dx- 2 \int \Re (u u')  x^{3+2s} \, x^{-n-1} \, dx,
$$
hence if $s<(n-2)/2,$
$$
(n-2-2s) \norm{x^{1+s} u}_{L^2_\mu}^2 \leq 2\big\lvert\ang{x^{1+s} u, x^{2+2s} u'}_{L^2_\mu}\big\rvert \leq
\lambda \norm{x^{1+s} u}_{L^2_\mu}^2 + \frac 1\lambda \norm{x^{2+s} u'}_{L^2_\mu}^2
$$
for all $\lambda>0.$
Thus if we also have $\lambda < n-2-2s,$
$$
\norm{x^{1+s}u}_{L^2_\mu}^2 \leq \frac{\norm{x^{2+s} u'}_{L^2_\mu}^2}{(n-2-2s)\lambda-\lambda^2}
$$
Optimizing by taking
$\lambda=(n-2)/2 -s$ yields the desired estimate.
\end{proof}

Since in a collar neighborhood of $\pa X,$
$$
\abs{\nabla_g u}^2_g \sim \abs{x\pa_\theta u}^2+ \abs{x^2 \pa_x u}^2,
$$
we can combine our Hardy inequality with the non-sharp Poincar\'e inequality
above to get a sharp result:

\begin{prop}\label{prop:sharp}
If $s<(n-2)/2$ and
$$u \in x^{-s+(n-2)/2}H^1_{\bl}(X),$$ then
$$
\norm{x^{1+s} u}^2_{L^2_g(X)}\leq C_s \norm{x^s \nabla u}_{L^2_g(X)}.
$$
In particular, the estimate holds for $u\in x^{-s}H^1_{\scl}(X),$ hence
for $u\in H^1_{\scl}(X)$ for $s\geq 0$.
\end{prop}
\begin{proof}
By density of $\dCI(X)$
and continutity of both sides in $x^{-s+(n-2)/2}H^1_{\bl}(X)$ (recall
that $L^2_g(X)=x^{n/2}L^2_{\bl}(X)$), it
suffices to consider $u\in\dCI(X)$ when proving the estimate.

Let $\phi \in \CI(X)$ equal $1$ on a collar neighborhood of $\pa X$ of
the form $\{x<\ep\}$ and equal $0$ on $\{x>2\ep\}.$  By integrating
the inequality Lemma~\ref{lemma:Hardy} in the angular variables,
i.e.\ along $\pa X$ in the collar neighborhood, we have
\begin{align*}
\norm{x^{1+s} \phi u}^2&\lesssim \norm{x^{2+s}
  \pa_x (\phi u)}^2 \\ &\lesssim \norm{x^s \nabla_g (\phi u)}^2\\ &\lesssim
\norm{x^s \nabla u}^2 + \norm{\phi' u}^2.
\end{align*}
(We use the notation $f \lesssim g$ to indicate that there exists $C
>0$ such that $|f| \leq Cg.$)
So overall we obtain
$$
\norm{x^{1+s} u}^2 \lesssim \norm{x^s \nabla u}^2 +
\norm{\phi' u}^2 + \norm{(1-\phi) u}^2.
$$
Now by compact support we certainly have
$$
\phi', (1-\phi) \lesssim x^{s+\ep},
$$
for all $\ep>0,$
hence by Lemma~\ref{lemma:Poincare} (with $l=n/2-s>1,$ $l'=n/2-s-\ep$),
$$
\norm{\phi' u}^2 + \norm{(1-\phi) u}^2 \lesssim \norm{x^s \nabla u}^2,
$$
and the desired estimate follows.
\end{proof}

Interpolating between  $\|x^s u\|_{L^2_g(X)}\leq \|x^s u\|_{L^2_g(X)}$ and
Proposition~\ref{prop:sharp}, we immediately deduce:

\begin{cor}
For $s<(n-2)/2$, $u\in x^{-s}H^1_{\scl}(X)$,
\begin{equation}\label{eq:Poincare-interpolate-gen}
  \|x^{s+\theta} u\|_{L^2_g(X)}\leq
C\|x^s\nabla_g u\|_{L^2_g(X)}^\theta\|x^s u\|_{L^2_g(X)}^{1-\theta},
\ 0\leq\theta\leq 1.
\end{equation}
In particular, if $n\geq 3$, $s=0$, then for $u\in H^1_{\scl}(X)$,
\begin{equation}\label{eq:Poincare-interpolate}
  \|x^{\theta} u\|_{L^2_g(X)}\leq
C\|\nabla_g u\|_{L^2_g(X)}^\theta\|u\|_{L^2_g(X)}^{1-\theta},
\ 0\leq\theta\leq 1.
\end{equation}
\end{cor}

Of course, we can estimate $\nabla_g u$ with a right side of similar form:
as $\Delta_g=\nabla_g^*\nabla_g$,
\begin{equation}\label{eq:nabla_g-est}
\|\nabla_g u\|^2_{L^2_g(X)}
=\langle \Delta_g u,u\rangle\leq \|\Delta_g u\|_{L^2_g(X)}
\|u\|_{L^2_g(X)}.
\end{equation}
Note also that if $Q\in x\Vb(X)=\Vsc(X)$ then
\begin{equation}\label{eq:Vsc-est}
\|Qu\|_{L^2_g(X)}\leq C\|\nabla_g u\|_{L^2_g(X)}.
\end{equation}

We can also consider $P=\Delta_g+V$, $V\in S^{-2-\rho}(X)$, $V\geq 0$,
$\rho>0$
as beforehand. Then
\begin{equation}\label{eq:nabla_g-V-est}
\|\nabla_g u\|^2_{L^2_g(X)}
=\langle \Delta_g u,u\rangle\leq
\langle (\Delta_g+V) u,u\rangle \leq \|(\Delta_g+V) u\|_{L^2_g(X)}
\|u\|_{L^2_g(X)}.
\end{equation}

\section{Weighted estimates for $\Delta_g+V$}\label{sec:weights}
We assume
throughout this section that $n\geq 3$, $g$ is a scattering metric
in the sense
of \eqref{eq:asymp-Eucl} with $g_1$ satisfying \eqref{eq:weaker-asymp},
$V\in S^{-2-\rho}(X)$, $V\geq 0$,
$\rho>0$. As below only $L^2_g(X)$ is of interest, we will write
$L^2(X)=L^2_g(X)$ henceforth.

For $0\leq s\leq 1$, $u\in\dCI(X)$, we now compute
\begin{equation}\begin{split}\label{eq:weighted-nabla}
\|x^s\nabla_g u\|^2_{L^2(X)}&=\langle\nabla_g u,x^{2s}\nabla_g u\rangle
=\langle\Delta_g u,x^{2s}u\rangle
+\langle\nabla_g u,[\nabla_g,x^{2s}]u\rangle\\
&=\langle(\Delta_g+V) u,x^{2s}u\rangle-\langle Vu,x^{2s}u\rangle
+\langle\nabla_g u,[\nabla_g,x^{2s}]u\rangle.
\end{split}\end{equation}
Now, for $0\leq s\leq 1/2$,
\begin{equation}\begin{split}\label{eq:weighted-nabla-calc-2}
|\langle(\Delta_g+V) u,x^{2s} u\rangle|&\leq \|(\Delta_g+V) u\|_{L^2(X)}
\|x^{2s} u\|_{L^2(X)}\\
&\leq C\|(\Delta_g+V) u\|_{L^2(X)} \|\nabla_g u\|^{2s}_{L^2(X)}
\|u\|^{1-2s}_{L^2(X)},
\end{split}\end{equation}
where we used \eqref{eq:Poincare-interpolate}.
On the other hand
$$[\nabla_g,x^{2s}]=x^{2s+1}f,\ f\in\CI(X;TX),$$
and $\sup|f|\leq C_0 s$, so using Proposition~\ref{prop:sharp} and $n\geq 3$
\begin{equation*}\begin{split}
|\langle\nabla_g u,[\nabla_g,x^{2s}]u\rangle|
&\leq C_0s\|x^s\nabla_g u\|_{L^2(X)} \|x^{s+1}u\|_{L^2(X)}\\
&\leq C_0 Cs\|x^s\nabla_g u\|_{L^2(X)} \|x^s\nabla_g u\|_{L^2(X)}
=C_0Cs\|x^s\nabla_g u\|_{L^2(X)}^2,
\end{split}\end{equation*}
and for $s$ sufficiently small this can be absorbed into
the left hand side of \eqref{eq:weighted-nabla}. Since
$\langle Vu,x^{2s}u\rangle\geq 0$, we deduce from
\eqref{eq:weighted-nabla} that there exists $s_0>0$ such
that for $0\leq s\leq s_0$,
\begin{equation}\label{eq:weighted-nabla-result}
\|x^s\nabla_g u\|^2_{L^2(X)}
\leq C\|(\Delta_g+V) u\|_{L^2(X)} \|\nabla_g u\|^{2s}_{L^2(X)}
\|u\|^{1-2s}_{L^2(X)};
\end{equation}
indeed this holds even with $\langle Vu,x^{2s}u\rangle$ added
to the left hand side. Although we had assumed
$u\in\dCI(X)$, by density and continuity, the estimate
holds for $u\in H^2_{\scl}(X)$, i.e.\ for $u$ in the domain of
$\Delta_g+V$.
Using the Poincar\'e inequality, Proposition~\ref{prop:sharp}, we
 deduce
\footnote{If $n\geq 5$, one can use Lemma~\ref{lemma:Vb-est} (or its analogue
if $V\geq 0$) to
obtain an estimate that slightly shortens some of the arguments that follow;
one then
needs to rely on the lemma, i.e.\ on b-machinery.
Namely, by Lemma~\ref{lemma:Vb-est}, if $n \geq 5,$
\begin{equation}\label{eq:b-estimates}
\|xQ_i u\|_{L^2(X)}\leq C\|\Delta_g u\|_{L^2(X)},
\ \|x^2 u\|_{L^2(X)}\leq C\|\Delta_g u\|_{L^2(X)}.
\end{equation}
On the other hand,
\begin{equation}\label{eq:sc-estimate}
\|Q_i u\|_{L^2(X)}\leq C\|\nabla_g u\|_{L^2(X)}.
\end{equation}
Interpolating between the first inequality of \eqref{eq:b-estimates}
and \eqref{eq:sc-estimate} gives for $n\geq 5$
\begin{equation}\label{eq:sc-b-interpolate}
\|x^s Q_i u\|_{L^2(X)}\leq C\|\nabla_g u\|^{1-s}_{L^2(X)}
\|\Delta_g u\|^s_{L^2(X)},\ 0\leq s\leq 1.
\end{equation}}:

\begin{prop}\label{prop:weighted-b-estimate-weak}
There exists $s_0>0$ such that for $0\leq s\leq s_0$
\begin{equation}\begin{split}\label{eq:weighted-combined-result-weak}
&\|x^{s+1} u\|_{L^2(X)}+\|x^s\nabla_g u\|_{L^2(X)}\\
&\qquad\leq C_s\|(\Delta_g+V) u\|_{L^2(X)}^{1/2} \|\nabla_g u\|^{s}_{L^2(X)}
\|u\|^{1/2-s}_{L^2(X)},\ u\in H^2_{\scl}(X).
\end{split}\end{equation}
In particular, for $L\in S^{-1-s}\Diffb^1(X)$,
\begin{equation}\label{eq:weighted-b-result-weak}
\|L u\|_{L^2(X)}
\leq C_s\|(\Delta_g+V) u\|_{L^2(X)}^{1/2} \|\nabla_g u\|^{s}_{L^2(X)}
\|u\|^{1/2-s}_{L^2(X)},\ u\in H^2_{\scl}(X).
\end{equation}
\end{prop}

Since any $L\in S^{-2-2s}\Diffb^2(X)$ can be rewritten as
$L=\sum Q_i^*R_i$, $Q_i,R_i\in S^{-1-s}\Diffb^1(X)$, with the sum
finite, we immediately deduce

\begin{cor}\label{cor:weighted-b-pairing-weak}
Let $s_0>0$ be as in Proposition~\ref{prop:weighted-b-estimate}.
For $0\leq s\leq s_0$, $L\in S^{-2-2s}\Diffb^2(X)$,
\begin{equation}\label{eq:weighted-pairing-estimate-weak}
|\langle Lu,u\rangle|\leq
C_s\|(\Delta_g+V) u\|_{L^2(X)} \|\nabla_g u\|^{2s}_{L^2(X)}
\|u\|^{1-2s}_{L^2(X)},\ u\in H^2_{\scl}(X).
\end{equation}
\end{cor}

In fact we can improve upon these results by allowing the
full range $0\leq s<(n-2)/2$
as follows. Rather than working
with $\langle x^{2s}(\Delta_g+V)u,u\rangle$, and rewriting it
in terms of $\|x^s\nabla_g u\|^2_{L^2(X)}$ plus a commutator, we
work with a symmetric expression:
$$
\langle f(\Delta_g+V)u,u\rangle+\langle u,f(\Delta_g+V)u\rangle
$$
for some $f$ which behaves like $x^{2s}$ for small $x$. First we compute
$$
f(\Delta_g+V)+(\Delta_g+V)f=2\nabla_g^* f\nabla_g+((\Delta_g+2V)f),
$$
where the last term on the right hand side is multiplication by
the function $(\Delta_g+2V)f$,
which can be seen by observing that both sides are real self-adjoint
second order scalar differential operators with the same principal symbol, so
their difference is first order, hence by reality and self-adjointness
zeroth order, and it vanishes on the constant function
$1$. Now if $f\geq 0$ then $Vf\geq 0$,
so all terms on the right hand side are positive provided $\Delta_g f\geq 0$,
and we have
\begin{equation*}\begin{split}
&2\langle f\nabla_g u,\nabla_g u\rangle +\langle ((\Delta_g f)u,u\rangle
+\langle 2Vf u,u\rangle\\
&\qquad=\langle f(\Delta_g+V)u,u\rangle+\langle u,f(\Delta_g+V)u\rangle,
\end{split}\end{equation*}
hence
\begin{equation}\label{eq:real-part-weighted-est}
\|f^{1/2}\nabla_g u\|^2\leq \|(\Delta_g+V)u\|\,\|fu\|.
\end{equation}
It remains to find $f\geq 0$ such that $\Delta_g f\geq 0$; we remind the
reader that this is the {\em positive} Laplacian.

With $t_0>0$ to be fixed, we consider
\begin{equation*}\begin{split}
&\chi(t)=e^{1/(t-t_0)},\ t<t_0,\\
&\chi(t)=0,\ t\geq t_0,
\end{split}\end{equation*}
and define $f$ by
\begin{equation}\label{eq:weight-def}
\begin{aligned}
f(p)&=\g(p)^{2s},\text{where}\\ \g(p)&=\chi(0)-\chi(x(p)/\ep),\ p\in X,
\end{aligned}
\end{equation}
where $\ep>0$. For $\ep>0$ is sufficiently small, $d\chi$ is supported
in such a collar neighborhood of $\pa X$ in which we can take $x$ as one of the
coordinates and $g$ is of the form \eqref{eq:asymp-Eucl} with $g_1$
as in \eqref{eq:weaker-asymp}. Moreover, $\g\geq 0$ (hence $f\geq 0$),
$\g'(0)>0$, and $\pa_x\g(x=0)=0$, hence $f\sim x^{2s}$ for $x$ near $0$.
As usual, we abuse notation and
write $f=f(x)$. Recall that
$$
\Delta_g =x^2\Delta_{\bl},\ \Delta_{\bl}=-(x\pa_x)^2+(n-2)(x\pa_x)+ \Delta_0
+x^\rho R,
\ R\in S^0\Diffb^2(X),
$$
and $R$ annihilates constants.
We then compute, for $x/\ep<t_0$ (since $df=0$ for $x/\ep\geq t_0$),
i.e.\ with $t=x/\ep$ for $0\leq t<t_0$, writing $f(p)=\g(p)^{2s}$, and primes
denoting derivatives in $t$,
\begin{equation*}\begin{split}
&\big(-x^2\pa_x^2+(n-3)x\pa_x\big)f=\big(-t^2\pa_t^2+(n-3)t\pa_t\big)\g^{2s}\\
&\qquad
=2s \g^{2s-2}\Big(-(2s-1)t^2 (\g')^2+(n-3)t\g\g'-t^2\g \g''\Big).
\end{split}\end{equation*}
Now, for $0\leq t<t_0$,
\begin{equation*}\begin{split}
&\g'=(t-t_0)^{-2} e^{1/(t-t_0)}>0,\\
&\g''=(t-t_0)^{-4}\big(-1-2(t-t_0)\big)e^{1/(t-t_0)}.
\end{split}\end{equation*}
We deduce that for $t_0<1/2$, $\g''< 0$ (on $[0,t_0)$).
Thus, for $n\geq 3$, $0<s<1/2$,
\begin{equation*}\begin{split}
&\big(-(x\pa_x)^2+(n-2)x\pa_x+\Delta_0\big)f\\
&\qquad=2s \g^{2s-2}\Big(-(2s-1)t^2 (\g')^2+(n-3)t\g\g'-t^2\g \g''\Big)\geq 0,
\end{split}\end{equation*}
i.e.\ the `model Laplacian' of $f$ is always non-negative provided $s\leq 1/2.$

If $n=3$, we have obtained non-negativity of $\Delta_g f$ for
the whole range $0<s<(n-2)/2$.  In general, however if $n>3$ and
$s\geq 1/2$, we need to estimate $t\g'$ relative to $\g$.  An estimate
$t\g'\leq C\g$ is automatic for sufficiently large $C>0$, as it is
easily checked at $0$, and $\g$ is bounded away from $0$ elsewhere.
However, we need a sharp constant, so we proceed as follows.  A
straightforward calculation gives
$$
t\g'-\g=\Big(\frac{t}{(t-t_0)^2}+1\Big)e^{1/(t-t_0)}-e^{-1/t_0},
$$
so $t\g'-\g$ vanishes at $t=0$ and it is decreasing, as its derivative is
$$
\frac{t}{(t-t_0)^4}\big(-1-2(t-t_0)\big)e^{1/(t-t_0)}\leq 0,\ 0\leq t<t_0,
\ t_0<1/2,
$$
so $t\g'\leq \g$ on $[0,t_0)$.
In summary
\begin{equation*}\begin{split}
\big(-t^2\pa_t^2+(n-3)t\pa_t\big)\g^{2s}
=2s \g^{2s-2}\Big((n-2s-2)t\g'\g-t^2\g\g''\Big)\geq 0,
\end{split}\end{equation*}
provided $1/2\leq s<(n-2)/2$, so
$$
\Big(-(x\pa_x)^2+(n-2)(x\pa_x)+ \Delta_0\Big)f\geq 0
$$
in this case.

We deduce that for any $0<s<(n-2)/2$ we have
$$
\Big(-(x\pa_x)^2+(n-2)(x\pa_x)+ \Delta_0\Big)f\geq 0,
$$
provided that we choose $0<t_0<1/2$, and indeed we have the
somewhat stronger estimate (useful for error terms below) that
for $c>0$ sufficiently small,
\begin{equation}\label{eq:positive-Delta-lb}
\Big(-(x\pa_x)^2+(n-2)(x\pa_x)+ \Delta_0\Big)f\geq
c\,\g^{2s-2}\Big(t^2(\g')^2-t^2\g\g''\Big),
\end{equation}
where both summands on the right hand side are non-negative, and where
we used $t\g'\leq \g$ in the case $s\geq 1/2$.
Note that this estimate is valid for any choice of $\ep>0$
provided it is sufficiently small (i.e.\ $\ep\leq\ep_1$, $\ep_1$ suitably
chosen) so that $df$ is supported in
the collar neighborhood of $\pa X$.
We can also deal with the error term $x^\rho R$ by letting $\ep\to 0$.
Namely, on the support of $Rf$, $x\leq\ep$, so $x^\rho R f\leq\ep^\rho Rf$,
so $\Delta_g f\geq 0$ follows provided
\begin{equation}\label{eq:Rf-needed-est}
Rf\leq C\Big(-(x\pa_x)^2+(n-2)(x\pa_x)\Big)f
\end{equation}
for some $C>0$.
But writing out $Rf$ explicitly in terms of $x\pa_x$ and $\pa_{y_j}$
in local coordinates (of which the latter annihilate $f$),
using that $R$ annihilates constants, we conclude that for
$C'>0$ sufficiently large
$Rf$ is bounded by
$$
C' \g^{2s-2}(t^2(\g')^2+t\g\g'-t^2 \g\g''),
$$
where we note that all terms in the parantheses are non-negative and
$C'$ is independent of $\ep\in (0,\ep_1]$. We now note that sufficiently
close to $0$, $t\g\g'$ can be absorbed into $t^2(\g')^2$ (uniformly in $\ep$)
for both are quadratic in $t$, and the latter is non-degenerate, while
outside any neighborhood of $0$, $t\g\g'$ can be absorbed in $-t^2\g\g''$,
i.e.\ $\g'$ can be absorbed into $\g''$, as is easy to check.
Thus, for $C''>0$ sufficiently large, $Rf$ is bounded by
$$
C'' \g^{2s-2}(t^2(\g')^2-t^2 \g\g''),
$$
and this is bounded by
$C'''\Big(-(x\pa_x)^2+(n-2)(x\pa_x)+\Delta_0\Big)f$ for sufficiently large
$C'''>0$ by \eqref{eq:positive-Delta-lb}, i.e.\ \eqref{eq:Rf-needed-est}
holds. This proves that for
$\ep>0$ sufficiently small $\Delta_g f\geq 0$. In summary we have proved:

\begin{lemma}\label{lemma:positive-Laplacian}
Let $0<t_0<1/2$, $0<s<(n-2)/2$.
Then there exists $\ep_0>0$ such that for $0<\ep<\ep_0$,
with $f$ as in \eqref{eq:weight-def}, $\Delta_g f\geq 0$.
\end{lemma}

As immediate consequences of Lemma~\ref{lemma:positive-Laplacian} and
\eqref{eq:real-part-weighted-est} we deduce that
\begin{equation}\label{eq:real-part-weighted-est-mod}
\|x^s\nabla_g u\|^2\leq C_s\|(\Delta_g+V)u\|\,\|x^{2s}u\|,
\end{equation}
which yields, in view of the Poincar\'e inequality, for $0\leq s<1/2$,
\begin{equation}\label{eq:real-part-weighted-est-Poincare}
\|x^s\nabla_g u\|^2
\leq C'_s \|(\Delta_g+V)u\| \|\nabla_g u\|^{2s}_{L^2(X)}
\|u\|^{1-2s}_{L^2(X)}.
\end{equation}
Using the Poincar\'e inequality again, and applying \eqref{eq:nabla_g-V-est}
we therefore deduce the
following strengthening of Proposition~\ref{prop:weighted-b-estimate-weak}:

\begin{prop}\label{prop:weighted-b-estimate}
For $0\leq s<1/2$
\begin{equation}\begin{split}\label{eq:weighted-combined-result}
&\|x^{s+1} u\|_{L^2(X)}+\|x^s\nabla_g u\|_{L^2(X)}\\
&\qquad\leq C_s\|(\Delta_g+V) u\|_{L^2(X)}^{1/2} \|\nabla_g u\|^{s}_{L^2(X)}
\|u\|^{1/2-s}_{L^2(X)}\\
&\qquad\leq C_s\|(\Delta_g+V) u\|_{L^2(X)}^{(1+s)/2}
\|u\|^{(1-s)/2}_{L^2(X)},\ u\in H^2_{\scl}(X).
\end{split}\end{equation}
In particular, for $L\in S^{-1-s}\Diffb^1(X)$,
\begin{equation}\begin{split}\label{eq:weighted-b-result}
\|L u\|_{L^2(X)}
&\leq C_s\|(\Delta_g+V) u\|_{L^2(X)}^{1/2} \|\nabla_g u\|^{s}_{L^2(X)}
\|u\|^{1/2-s}_{L^2(X)}\\
&\leq C_s\|(\Delta_g+V) u\|_{L^2(X)}^{(1+s)/2}
\|u\|^{(1-s)/2}_{L^2(X)},\ u\in H^2_{\scl}(X).
\end{split}\end{equation}
\end{prop}

Using again that any $L\in S^{-2-2s}\Diffb^2(X)$ can be rewritten as
$L=\sum Q_i^*R_i$, $Q_i,R_i\in S^{-1-s}\Diffb^1(X)$, with the sum
finite, we conclude

\begin{cor}\label{cor:weighted-b-pairing}
For $0\leq s<1/2$, $L\in S^{-2-2s}\Diffb^2(X)$,
\begin{equation}\begin{split}\label{eq:weighted-pairing-estimate}
|\langle Lu,u\rangle|&\leq
C_s\|(\Delta_g+V) u\|_{L^2(X)} \|\nabla_g u\|^{2s}_{L^2(X)}
\|u\|^{1-2s}_{L^2(X)}\\
&\leq
C_s\|(\Delta_g+V) u\|_{L^2(X)}^{1+s}\|u\|^{1-s}_{L^2(X)},\ u\in H^2_{\scl}(X).
\end{split}\end{equation}
\end{cor}

Now, suppose that $u=\psi(H^2(\Delta_g+V))v$, $v\in L^2(X)$,
where $\psi\in L^\infty_c(I)$,
$I\subset (0,\infty)$ compact, $0\leq\psi\leq 1$, $H>0$. Then
$u\in\Hsc^2(X)$ and
$$
C'_I\|u\|_{L^2(X)}\leq\|H^2(\Delta_g+V) u\|_{L^2(X)}\leq C_I\|u\|_{L^2(X)}
$$
and
$$
C'_I\|u\|^2_{L^2(X)}\leq\langle H^2(\Delta_g+V) u,u\rangle\leq C_I\|u\|^2_{L^2(X)}.
$$
Combining these with Corollary~\ref{cor:weighted-b-pairing}
we deduce that
for $L\in S^{-2-\sigma}\Diffb^2(X)$ with $0\leq\sigma<1$,
\begin{equation}\label{eq:localized-weighted-b-est}
|\langle Lu,u\rangle|\leq C' C_I^{1+\sigma/2} H^{-2-\sigma}\|u\|^2_{L^2(X)}.
\end{equation}
Note that $|\langle Vu,u\rangle|$ satisfies the same estimate as
$|\langle Lu,u\rangle|$. If $\sigma>0$ this gives a gain of $H^{-\sigma}$
over e.g.\ $\langle(\Delta+V)u,u\rangle$ as $H\to\infty$; ultimately,
this gain arose due to
the Poincar\'e estimate in \eqref{eq:weighted-nabla-calc-2}.
We also remark that \eqref{eq:weighted-b-result} yields for $L\in S^{-1-s}\Diffb^1(X)$, $0\leq s<1/2$,
\begin{equation}\label{eq:localized-weighted-b-est-2}
\|L u\|_{L^2(X)}
\leq C C_I^{(1+s)/2} H^{-1-s}\|u\|_{L^2(X)}.
\end{equation}
The estimates
\eqref{eq:localized-weighted-b-est}--\eqref{eq:localized-weighted-b-est-2}
are analogues of Lemma~B.12 of \cite{Bony-Haefner1}, with $\lambda=H^2$
in their notation: one can trade powers of $x$ for negative powers of $H$
(within limits),
i.e.\ in the notation of \cite{Bony-Haefner1}, one can trade
negative powers of $\langle x\rangle$ for powers of $\lambda^{-1/2}$
(see the exponent $\gamma$ in \cite{Bony-Haefner1}).

Below we actually need a somewhat stronger result, using the resolvent in
place of the compactly supported functions of $P=\Delta_g+V$.
Thus, for $L\in S^{-1-s}\Diffb^1(X)$, $u\in L^2(X)$, replacing
$u$ by $(\Delta_g+V-z)^{-1}u$, and using
$\|(\Delta_g+V-z)^{-1}\|_{\cL(L^2(X))}\leq |\im z|^{-1}$ (for $\im z\neq 0$),
we deduce that
\begin{equation}\begin{split}\label{eq:res-weight-gain}
\|&L (\Delta+V-z)^{-1}u\|_{L^2(X)}\\
&\leq C\|(\Id+z(\Delta_g+V-z)^{-1})u\|_{L^2(X)}^{(1+s)/2}
\|(\Delta_g+V-z)^{-1}u\|^{(1-s)/2}_{L^2(X)}\\
&\leq C(1+|z|/|\im z|)^{(1+s)/2}\|u\|_{L^2(X)}^{(1+s)/2}
\,|\im z|^{-(1-s)/2}\|u\|^{(1-s)/2}_{L^2(X)}\\
&\leq 2C(|z|/|\im z|)^{(1+s)/2}|\im z|^{-(1-s)/2}\|u\|_{L^2(X)}.
\end{split}\end{equation}
In addition, using the positivity of $\Delta_g+V$, we have for $z$ with
$\re z<0$,
$$
\|(\Delta_g+V-z)^{-1}\|_{\cL(L^2(X))}\leq |z|^{-1},
$$
so
in fact
\begin{equation}\label{eq:res-weight-gain-neg}
\|L (\Delta+V-z)^{-1}u\|_{L^2(X)}\leq 2C|z|^{-(1-s)/2}\|u\|_{L^2(X)},\ \re z<0.
\end{equation}
Replacing $z$ by $z=w/H^2$, we deduce the following:

\begin{prop}
Suppose $L\in S^{-1-s}\Diffb^1(X)$, $0\leq s<1/2$. Then there exists
$C>0$ such that for all $u\in L^2(X)$ we have
\begin{equation}\begin{split}\label{eq:res-weight-gain-H}
&\|L (H^2(\Delta+V)-w)^{-1}u\|_{L^2(X)}\\
&\qquad\qquad\qquad
\leq 2CH^{-1-s}(|w|/|\im w|)^{(1+s)/2}|\im w|^{-(1-s)/2}\|u\|_{L^2(X)},
\ \im w\neq 0,\\
&\|L (H^2(\Delta+V)-w)^{-1}u\|_{L^2(X)}\leq 2CH^{-1-s}
|w|^{-(1-s)/2}\|u\|_{L^2(X)},\ \re w<0.
\end{split}\end{equation}
\end{prop}

In particular, this gives
uniform bounds on $L(H^2(\Delta+V)-w)^{-1}$ in $\cL(L^2(X))$.

\section{Low frequency Mourre estimate}\label{sec:Mourre}
We now prove the low frequency Mourre estimate.

Let $\phi\in\CI_c(X)$ be chosen as above, i.e.\ let it be identically $1$ near $\pa X$, supported
in a collar neighborhood of $\pa X$, on which $x D_x$ is thus defined,
and let
$$
A=-\frac{1}{2}\big((\phi xD_x)+(\phi xD_x)^*\big).
$$
Since (cf.\ \eqref{eq:Delta-b-form})
$$
\Delta_g=\sum Q_i^* G_{ij} Q_j
=(x^2D_x)^*(x^2 D_x)+x^2 d_{\pa X}^*d_{\pa X}+x^{2+\rho}R,
$$
where
$$
Q_i\in\Vsc(X),\ G_{ij}\in S^0(X),\ R\in S^0\Diffb^2(X),
$$
we have
\begin{equation}\begin{split}\label{eq:comm-calc}
&[\Delta_g+V,A]=-2i(\Delta_g+L),\ L\in S^{-2-\rho}\Diffb^2(X),
\end{split}\end{equation}
at first as a quadratic form on $\dCI(X)$, but then noting that the
right hand side extends (by density) to a continuous map from
$H^2_{\scl}(X)$ to $L^2(X)$.
For $u=\psi(H^2(\Delta_g+V))v$, $v\in L^2(X)$, we
now use Corollary~\ref{cor:weighted-b-pairing}.
Thus, without loss of generality taking $\rho<1,$
\eqref{eq:localized-weighted-b-est} gives
\begin{equation*}
|\langle Lu,u\rangle|\leq C' C_I^{1+\rho/2} H^{-2-\rho}\|u\|^2_{L^2(X)}.
\end{equation*}
Note that $|\langle Vu,u\rangle|$ satisfies the same estimate as
$|\langle Lu,u\rangle|$.

In summary,
\begin{equation*}\begin{split}
&\langle \frac{i}{2}[\Delta_g+V,A]u,u\rangle
=\Big\langle \Big(\Delta_g+V+L-V\Big)u,u\Big\rangle\\
&\qquad\geq \langle(\Delta_g+V) u,u\rangle
-CH^{-2-\rho}\|u\|_{L^2(X)}^2
=H^{-2}\langle (H^2(\Delta_g+V)-CH^{-\rho})u,u\rangle.
\end{split}\end{equation*}
We thus deduce that there exist $H_0>0$ and $C'>0$ such that for $H>H_0$,
\begin{equation*}
\langle \frac{i}{2}[H^2(\Delta_g+V),A]u,u\rangle
\geq C'\|u\|^2,\ u=\psi(H^2(\Delta_g+V))v.
\end{equation*}
Now let $\psi=\chi_I$, the characteristic function of $I$, we deduce the
following:

\begin{thm}
Suppose $n\geq 3$, $g$ is a scattering metric in the sense
of \eqref{eq:asymp-Eucl} with $g_1$ satisfying \eqref{eq:weaker-asymp},
$V\in S^{-2-\rho}(X)$, $\rho>0$, $V\geq 0$, $P=\Delta_g+V$.
Let $I\subset(0,\infty)$ be a compact interval, and $\chi_I$ the
characteristic function of $I$. Then there exist
$H_0>0$ and $C>0$ such that for
$H>H_0$,
\begin{equation*}
\chi_I(H^2 P)\frac{i}{2}[H^2 P,A]\chi_I(H^2 P)
\geq C\chi_I(H^2 P).
\end{equation*}
In particular for $\psi\in\CI((0,\infty))$,
\begin{equation}\label{eq:weight-factor-Mourre}
\psi(H^2 P)\chi_I(H^2 P)\frac{i}{2}[H^2 P,A]
\chi_I(H^2 P)\psi(H^2 P)
\geq C(\inf_I\psi)^2\chi_I(H^2 P).
\end{equation}
\end{thm}

\begin{rem}
The commutator is defined here as a quadratic form on $\dCI(X)$, which
extends to $H_{\scl}^2(X)$ continuously.
If $\psi\in\CI_c(I)$ then for $v\in\dCI(X)$ one has $u\in\dCI(X)$ by
the functional calculus in the algebra of scattering pseudodifferential
operators -- the main point here is that the decay properties are preserved,
see \cite[Theorem~11]{Hassell-Vasy:Symbolic}.
(One can also obtain this decay without using the full ps.d.o.\ algebra,
working with the Helffer-Sj\"ostrand formula and commutators directly,
if one so desires.) Thus, for such $v$ and $\psi$, one can expand the
commutator and manipulate it directly, which is important in applications.
\end{rem}

This at once implies the corresponding estimate with $H^2 P$
replaced by $H\sqrt{ P}$, which is the main content of
\cite[Proposition~3.1]{Bony-Haefner1} when $X=\RR^n$ equipped with a
metric asymptotic to the standard Euclidean metric. In order to do
this recall that $\Hsc^{m,l}(X)=x^l\Hsc^m(X)$ is the scattering
Sobolev space of Melrose \cite{RBMSpec}, which for $X$ the radial
compactification of $\RR^n$ is just the standard weighted Sobolev
space $H^{m,l}(\RR^n)$, and one has the high energy estimate that
$(P+\lambda)^{-1}:\Hsc^{m,l}(X)\to\Hsc^{m,l}(X)$, $P=\Delta_g+V$, is bounded by
$C\lambda^{-1}$ in $\lambda>1$ from the semiclassical scattering
calculus; this is of course very easy to see for $l=0$, which is what
we need below.  (Recall that we are using the nonnegative Laplace
operator.)  Now, one has by the functional calculus
$$
\sqrt{P}=\pi^{-1}\int_0^\infty \lambda^{-1/2}
P(P+\lambda)^{-1}\,d\lambda,
$$
so
$$
H\sqrt{P}=\pi^{-1}\int_0^\infty \lambda^{-1/2}
H^2P(H^2P+\lambda)^{-1}\,d\lambda;
$$
using the above observation,
the integral converges for any $m$ as a bounded operator in
$\cL(\Hsc^{m,0}(X),\Hsc^{m-2,0}(X))$.
We now evaluate the commutator $[H\sqrt{P},A]
:\Hsc^{m,1}(X)\to\Hsc^{m-3,-1}(X)$; the integral for the products
$H\sqrt{P}A$ and $A H\sqrt{P}$ converges in this sense.
As
$$
[H^2P(H^2P+\lambda)^{-1},A]
=\lambda (H^2P+\lambda)^{-1}[H^2P,A](H^2P+\lambda)^{-1},
$$
using $(t+\lambda)^{-1}\geq (\sup I+\lambda)^{-1}$ on $I$,
we deduce from \eqref{eq:weight-factor-Mourre} that for $H>H_0$,
\begin{equation}\begin{split}\label{sqrtcommutator}
&\chi_I(H^2P)[H\sqrt{P},A]\chi_I(H^2P)\\
&\quad=\pi^{-1}\int_0^\infty \lambda^{1/2}
(H^2P+\lambda)^{-1}
\chi_I(H^2P)[H^2P,A]\chi_I(H^2P)(H^2P+\lambda)^{-1}
\,d\lambda\\
&\quad\geq \pi^{-1}\int_0^\infty C\lambda^{1/2}(\sup I+\lambda)^{-2}
\chi_I(H^2P)^2\,d\lambda=C'\chi_I(H^2P)^2,\ C'>0,
\end{split}\end{equation}
on $\Hsc^{3,1}(X)$, hence by density of $\Hsc^{3,1}(X)$ and continuity
of both sides on $L^2(X)$, on $L^2(X)$.
Thus, the analogue of the
low energy Mourre estimate of Bony and H\"afner in this more general
setting follows immediately.

\begin{thm}
Suppose $n\geq 3$, $g$ is a scattering metric in the sense
of \eqref{eq:asymp-Eucl} with $g_1$ satisfying \eqref{eq:weaker-asymp},
$V\in S^{-2-\rho}(X)$, $\rho>0$, $V\geq 0$, $P=\Delta_g+V$.
Let $I\subset(0,\infty)$ be a compact interval, and $\chi_I$ the
characteristic function of $I$. Then there exist
$H_0>0$ and $C>0$ such that for
$H>H_0$,
\begin{equation*}
\chi_I(H^2P)\frac{i}{2}[H\sqrt{P},A]\chi_I(H^2P)
\geq C\chi_I(H^2P).
\end{equation*}
\end{thm}

\section{Energy decay for the wave equation}\label{sec:wave}
If the metric on $X$ is additionally assumed to be non-trapping, we
have a finite- and high-energy Mourre estimate due to Vasy-Zworski
\cite{Vasy-Zworski} (or can re-use the construction employed in
\cite{Bony-Haefner1}).  Putting these ingredients together as in
Theorem~1.3 of \cite{Bony-Haefner1} we obtain by the same means the
analogous energy decay result for solutions to the wave equation; for
brevity, we confine our discussion of these results to the case of
(unperturbed) scattering metrics, i.e.\ those given by
\eqref{eq:asymp-Eucl} near infinity.
\begin{thm}
  Let $(X,g)$ be a scattering manifold having no trapped geodesics, and let $V \in
  S^{-3}(X)$ be a nonnegative potential.  If
$$
\big(D_t^2-(\Delta_g +V)\big)u=0
$$
on $\RR\times X,$ then for all $\ep>0$ and $\mu \in (0,1],$
$$
\norm{x^\mu u'}_{L^2([0,T] \times X)} \lesssim \ang{F_\mu^\ep(T)}^{1/2}
  \norm{u'(0,\cdot)}_{L^2(X)}
$$
where $u'=(\pa_t u, \nabla_g u).$
and
$$
F_\mu^\ep(T)= \begin{cases} T^{1-2\mu-2\ep}, & \mu\leq 1/2 \\1 & \mu>1/2\end{cases}.
$$
\end{thm}
\begin{proof}
As indicated above, the relevant medium and high energy estimates are well
known in this setting, and it will suffice, following the strategy of
\cite{Bony-Haefner1}, to demonstrate that the low energy commutant
that we have constructed above satisfies all of the hypotheses of the
Mourre theory.  As discussed in Proposition~3.1 of
\cite{Bony-Haefner1}, it remains for us to verify, in our notation,
the following estimates on the operator
$$
\A_H \equiv \psi(H^2P) A\psi(H^2P):
$$
\begin{align}
\label{ad1} \norm{[\A_H, H P^{1/2}]} &\lesssim 1,\\
\label{ad2}\norm{\big[\A_H,[\A_H,
  H P^{1/2}]\big]} &\lesssim 1,\\
\label{mourre1} \norm{\abs{\A_H}^\mu x^\mu} &\lesssim
H^{-\mu}, \quad \mu \in [0,1]\\
\label{mourre2} \norm{\ang{\A_H}^\mu\psi(H^2P) x^\mu} &\lesssim H^{-\mu},\quad \mu \in [0,1].
\end{align} 

We begin by proving a lemma allowing us to commute powers of $x$ with
spectral projections:
\begin{lemma}\label{lemma:conjugateprojector}
Let $L \in x \Diff^1_b(X).$
The operators
$$
x^{-1} \psi(H^2 P) L
\text{ and } L \psi(H^2 P) x^{-1}
$$
are uniformly $L^2$-bounded as $H \uparrow \infty.$
\end{lemma}
\noindent\emph{Proof of lemma:}
As the two types of operator in question are adjoints of one another, it
suffices to consider the latter.  Moreover, for the desired boundedness
it suffices to estimate $L  [\psi(H^2 P), x^{-1}].$

Letting $\tpsi$ be a compactly supported
almost-analytic extension of $\psi.$  Let $R(z)$ denote the
resolvent
$$
R(z)=(H^2P-z)^{-1}.
$$
We have
\begin{align*}
L [\psi(H^2 P), x^{-1}]&= \frac 1{2\pi} \int_{\CC} \overline{\pa} \tpsi(z)
L [R(z), x^{-1}]\, dz d\overline{z}\\ 
&= -\frac{H^2}{2\pi} \int_{\CC} \overline{\pa} \tpsi(z)
L R(z) [P, x^{-1}]   R(z) \, dz d\overline{z}.
\end{align*}
As $[P,x^{-1}] =Q\in x\Diff_b^1(X)$, \eqref{eq:res-weight-gain-H} gives
(with $s=0$)
that
\begin{equation*}\begin{split}
&\|LR(z)\|_{\cL(L^2(X))}\leq CH^{-1}|\im z|^{-1} |z|^{1/2},\\
&\|QR(z)\|_{\cL(L^2(X))}\leq CH^{-1} |\im z|^{-1} |z|^{1/2},
\end{split}\end{equation*}
so we can estimate the integral by a multiple of
$$
H^2\int_{\CC} \big\lvert \overline{\pa} \tpsi(z) \big \rvert H^{-2} \abs{\Im z}^{-2}|z|\, dz d\overline{z}.
$$
\emph{This concludes the proof of the lemma.}

We now sketch the proofs of \eqref{ad1}--\eqref{mourre2}.  The
estimate \eqref{ad1} follows from \eqref{sqrtcommutator}, as we may
again write
\begin{equation}\begin{split}\label{sqrtcommutator2}
&\psi(H^2P)[H\sqrt{P},A]\psi(H^2P)\\
&\quad=\pi^{-1}\int_0^\infty \lambda^{1/2}
R(\lambda)
\psi(H^2P)[H^2P,A]\psi(H^2P)R(\lambda)
\,d\lambda.
\end{split}\end{equation}
By anti-self-adjointness of the commutator, it suffices to estimate
the norm of 
\begin{equation*}\begin{split}
&\ang{[\A_H, H P^{1/2}] u,u}\\
& = \pi^{-1}\int_0^\infty \lambda^{1/2}
\ang{R(\lambda)
\psi(H^2P)[H^2P,A]\psi(H^2P)R(\lambda)u,u}
\,d\lambda.
\end{split}\end{equation*}
Now as $[H^2 P,A] \in x^2 \Diff^2_b(X),$ we may rewrite this
pairing in the form
$$
H^2 \int_0^\infty \lambda^{1/2}
\ang{x M_1
\psi(H^2P)R(\lambda) u , x M_2
\psi(H^2P)R(\lambda) u}
\,d\lambda
$$
with $M_i \in \Diff^1_b(X).$
Applying \eqref{eq:localized-weighted-b-est-2} (in the `easy' case
$s=0$) yields \eqref{ad1}.

We can now prove \eqref{ad2} in the same manner (cf.\ Remark~3.5 in
\cite{Bony-Haefner1}): By \eqref{eq:comm-calc} we have
$$
[A,P] =2i P +M,
$$
where
$$
M \in S^{-3} \Diff^2_b(X).
$$
Thus,
$$
[\A_H,HP] =2i P +M,
$$
Thus
\begin{equation}\begin{split}\label{sqrtcommutator3}
[\A_H, HP^{1/2}]&=\psi(H^2P)[A,H\sqrt{P}]\psi(H^2P)\\
&=\pi^{-1}\int_0^\infty \lambda^{1/2}
R(\lambda)
\psi(H^2P)[A,H^2P]\psi(H^2P)R(\lambda)\, d\lambda\\
&=\pi^{-1}\int_0^\infty \lambda^{1/2}
R(\lambda)
\psi(H^2P)(2iP+M)\psi(H^2P)R(\lambda)\, d\lambda\\
&=2i\psi(H^2P)^2 H P^{1/2}+
\pi^{-1}\int_0^\infty \lambda^{1/2}
R(\lambda)
\psi(H^2P) M \psi(H^2P)R(\lambda)\, d\lambda\\
&=2i\psi(H^2P)^2 H P^{1/2}+ \B
\end{split}\end{equation}
Hence to estimate
$$
[\A_H,[\A_H, HP^{1/2}]],
$$
by \eqref{ad1}, it suffices to estimate
$$
[\A_H, \B];
$$
to do this, noting that $\A_H$ is self-adjoint, and $\B$ is
anti-self-adjoint, we see that it suffices to obtain boundedness of
$$
\A_H \B = (\A_H x) (x^{-1} \B).
$$
Now application of Lemma~\ref{lemma:conjugateprojector} shows that
$$
\A_H x = (\psi(H^2 P) A x) (x^{-1}\psi(H^2P) x)
$$
is uniformly bounded. Likewise, $x^{-1} \B$ is bounded by similar
considerations: we write the integrand for $x^{-1} \B$ as
\begin{equation*}\begin{split}
&\lambda^{1/2} x^{-1} R(\lambda) \psi(H^2P) M \psi(H^2P)R(\lambda)\\
&\qquad=\sum_j
\lambda^{1/2}\big( x^{-1}
R(\lambda)\psi(H^2P) (xL_{j1})\big)  \big( L_{j2} \psi(H^2P)R(\lambda)\big).
\end{split}\end{equation*}
where $M= \sum_j xL_{j1} L_{j2},$ $L_{ji} \in x\Diff^1_b(X).$
By
\eqref{eq:localized-weighted-b-est-2} the last factor is norm bounded by a
multiple of $H^{-1} (c+\lambda)^{-1},$ $c =\inf \supp \psi>0.$ 
Commuting the factor of $x^{-1}$ across both $R(\lambda)$ and
$\psi(H^2 P)$ yields an operator bounded by a multiple of $H
\lambda^{-1}$ by \eqref{eq:localized-weighted-b-est-2}, while the
commutator terms involved in doing this have the same bound by
Lemma~\ref{lemma:conjugateprojector} and the observation that
$$
[x^{-1}, R(\lambda)] \psi(H^2P) = -R(\lambda) Q R(\lambda) \psi(H^2P)
$$
with $Q \in x \Diff^1_b(X);$ this expression is bounded by a multiple of $H^{-1}
(c+\lambda)^{-1}$ by \eqref{eq:res-weight-gain-H}.  We thus obtain \eqref{ad2}.

To prove \eqref{mourre1}, it suffices by interpolation to prove
uniform boundedness as $H\uparrow \infty$ of
$$
\norm{H\psi(H^2P) A \psi(H^2P) x}_{\cL(L^2(X))};
$$
as above, this follows from Lemma~\ref{lemma:conjugateprojector}.
Likewise, \eqref{mourre2} follows by interpolation with the $\mu=1$ estimate
$$
\norm{H\psi(H^2P) A \psi(H^2P) x}_{\cL(L^2(X))}^2 +\norm{\psi(H^2P) x}_{\cL(L^2(X))}^2.
$$
\end{proof}


\begin{thebibliography}{X}
\bibitem{Bony-Haefner1}
J.-F. Bony and  D. H\"afner, \emph{The semilinear wave equation on
  asymptotically Euclidean manifolds}, preprint, 2008.

\bibitem{Bony-Haefner2}
J.-F. Bony and  D. H\"afner,
\emph{Low frequency resolvent estimates for long range perturbations of the
Euclidean Laplacian}, preprint, 2009.

\bibitem{Bouclet1}
J.-M. Bouclet, \emph{Low energy behaviour of the resolvent of
  long range perturbations of the Laplacian}, preprint, 2009.

\bibitem{Guillarmou-Hassell:Resolvent-I}
C. Guillarmou and A. Hassell,
\emph{Resolvent at low energy and {R}iesz transform for
Schr{\"o}dinger operators on asymptotically conical manifolds. {I}},
Math. Ann. 341 (2008), no. 4, 859--896.

\bibitem{Hassell-Vasy:Symbolic}
A. Hassell and A. Vasy,
\emph{Symbolic functional calculus and {N}-body resolvent estimates},
J. Func. Anal. 173 (2000),257--283.

\bibitem{Mazzeo-McOwen}
R. Mazzeo and R. McOwen,
\emph{Singular Sturm-Liouville theory on manifolds},
J. Differential Equations 176 (2001), no. 2, 387--444. 

\bibitem{Melrose:Atiyah}
R.~B. Melrose,
\emph{The {A}tiyah-{P}atodi-{S}inger index theorem},
A K Peters Ltd., 1993.

\bibitem{RBMSpec}
R.~B. Melrose,
\emph{Spectral and scattering theory for the Laplacian on
asymptotically Euclidian spaces}. In \emph{Spectral and scattering theory},
M.~Ikawa, editor, Marcel Dekker, 1994.

\bibitem{Vasy-Zworski}
A. Vasy and M. Zworski,
\emph{Semiclassical estimates in asymptotically Euclidean scattering,}
Comm. Math. Phys. 212 (2000), no. 1, 205--217. 
\end{thebibliography}
\end{document}